\documentclass[12pt,a4paper]{article}
\usepackage[tbtags]{amsmath}    
\usepackage{amsfonts,amssymb,amsthm,euscript,makeidx,pifont,boxedminipage,multicol,fullpage}
\usepackage[T1]{fontenc}
\usepackage{setspace}
\usepackage{soul}   

\newtheorem{theorem}{Theorem}[section]
\newtheorem{lemma}[theorem]{Lemma}
\newtheorem{proposition}[theorem]{Proposition}

\newtheorem{remark}[theorem]{Remark}

\newcommand{\Bc}{\mathcal{B}}

\newcommand{\Ee}{\mathbb{E}}

\newcommand\R{\mathbb{R}}
\newcommand\Uc{\mathcal{U}}

\usepackage{color}

\newcommand\indi[1]{{\mathbb{I}}_{{#1}}}

\def \Om{\Omega}
\def \Omb{\overline{\Omega}}
\def \om{\omega}
\def \omb{\bar{\omega}}
\def \w{\mathsf{w}}
\def \wb{\bar{\w}}
\def \Fc{\mathcal{F}}
\def \Gc{\mathcal{G}}
\def \F{\mathbb{F}}

\def \Fb{\overline{\F}}
\def \Fcb{\overline{\Fc}}
\def \x{\times}
\def \ox{\otimes}
\def \xb{\mathbf{x}}
\def \yb{\mathbf{y}}
\def \0{\mathbf{0}}
\def \E{\mathbb{E}}

\def \P{\mathbb{P}}

\def \Pb{\overline{\P}}

\def \Nc{\mathcal{N}}
\def \Lc{\mathcal{L}}
\def \etab{\bar \eta}
\newcommand{\rmi}{{\rm (i)$\>\>$}}
\newcommand{\rmii}{{\rm (ii)$\>\>$}}
\newcommand{\rmiii}{{\rm (iii)$\>\,    \,$}}
\newcommand{\rmiv}{{\rm (iv)$\>\>$}}
\newcommand{\rmv}{{\rm (v)$\>\>$}}

\title{A pseudo-Markov property for controlled diffusion processes}

\author{Julien CLAISSE\thanks{INRIA Sophia Antipolis,
2004 route des Lucioles, BP93, Sophia Antipolis, France.
Email: julien.claisse@inria.fr}
	\and Denis TALAY\thanks{INRIA Sophia Antipolis,
2004 route des Lucioles, BP93, Sophia Antipolis, France.
Email: denis.talay@inria.fr}
	\and Xiaolu TAN\thanks{CEREMADE,
Universit\'e Paris--Dauphine,
Place du Mar\'echal De Lattre De Tassigny,
75775 Paris Cedex 16, France.
Email:	tan@ceremade.dauphine.fr}
}

\date{12-01-15}

\begin{document}
\bibliographystyle{plain}

\maketitle

\abstract{{In this note, we propose two different approaches to
	rigorously justify a pseudo-Markov property
	for controlled diffusion processes
which is often (explicitly or implicitly) used to prove the
	dynamic programming principle
	in the stochastic control literature.
	The first approach develops a sketch of proof proposed by Fleming and
	Souganidis~\cite{fleming-souganidis}.
	The second approach is based on an enlargement of
	the original state space and a controlled martingale
	problem. We clarify some measurability and topological issues
raised by these two approaches.
}
}

\vspace{1mm}


\vspace{2mm}


\section{Introduction}

The dynamic programming principle (DPP) is a key
step in the mathematical analysis of optimal stochastic
control problems. For various formulations and approaches,
among hundreds of references, see, e.g.,
	 Fleming and Rishel \cite{FlemingRishel},
	 Krylov~\cite{krylov}, El Karoui \cite{ElKarouiSF},
	 Borkar \cite{Borkar}, Fleming and
	 Soner~\cite{fleming-soner}, Yong and
	 Zhou~\cite{yong-zhou},
and the more recent monographs by Pham~\cite{Pham} or
Touzi~\cite{Touzi}.

The DPP has a very intuitive meaning
but its rigorous proof for controlled diffusion processes
is a difficult issue in all cases where
the set of admissible controls is the set of
all progressively
measurable processes taking values in a given domain.
In particular this occurs in the stochastic analysis of
Hamilton-Jacobi-Bellman equations.
The highly technical approach to the DPP
developed by Krylov in~\cite{krylov} 
consists in establishing a continuity
property of expected cost functions w.r.t. controls
(see~\cite[Cor.3.2.8]{krylov})
in order to be in a position
to only prove the DPP in the simpler case where the controls
are piecewise constant.
When the controls belong to this restricted set,
the DPP is obtained by proving that
the corresponding controlled diffusions enjoy
a pseudo-Markov property which easily results
from the classical Markov property enjoyed by
uncontrolled diffusions.
See ~\cite[Lem.3.2.14]{krylov}.

In their context of stochastic differential games problems, Fleming and
Souganidis~\cite[Lem.1.11]{fleming-souganidis} propose
to avoid the control approximation procedure
by establishing the pseudo-Markov property without
restriction on the controls. Although natural,
this way to establish the DPP leads to various difficulties.
 In particular,
one needs to restrict the state space to the standard canonical
space. Unfortunately the proof of the crucial lemma~1.11
is only sketched.

Many other authors had implicitly or explicitly followed
the same way and partially justified
the pseudo-Markov property in the case of general
admissible controls: see, e.g., Tang and
Yong~\cite[Lem.2.2]{tang-yong}, Yong
and Zhou~\cite[Lem.4.3.2]{yong-zhou}, Bouchard
and Touzi~\cite{bouchard-touzi} and Nutz~\cite{NutzEJP}.
However, to the best of our knowledge, there is no full
and precise proof available in the literature except in
restricted situations (see, e.g., Kabanov and
Kl\"{u}ppelberg~\cite{kabanov-kluppelberg}).

When completing the elements of proof provided by Fleming and Souganidis and the other above mentioned authors,
we found several subtle technical difficulties to overcome.
For example,
the stochastic integrals and solutions to stochastic differential equations
are usually defined on a complete probability space 
equipped with a filtration satisfying the usual conditions;
on the other hand, the completion of the canonical 
$\sigma$--field is undesirable when
one needs to use a family of regular conditional probabilities.
See more discussion in Section~\ref{subsec:discussion}.

We therefore find it useful to provide a precise formulation
and a detailed full justification of the pseudo-Markov
property for general controlled diffusion processes,
and to clarify measurability and topological
issues,
particularly the importance of
setting stochastic control problems in the
canonical space rather to an abstract non Polish 
probability space
if the pseudo-Markov property is used to prove
the DPP.

We provide two proofs.
The first one develops the arguments sketched by
Fleming and Souganidis~\cite{fleming-souganidis}.
The second one is simpler in some aspects but requires
some more heavy notions; it is based on 
an enlargement
of the original state space and a  suitable controlled
martingale problem.

The rest of the paper is organized as follows.
In Section~\ref{sec:strong_formulation}, we introduce
the classical formulation of controlled SDEs.
We then precisely state the pseudo-Markov property
for controlled diffusions, which is our main result.
We discuss its use to establish the DPP.
To prepare its two proofs, technical lemmas are proven
in Section~\ref{sec:technical_lemmas}.
Finally, in Sections~\ref{sec:proof1} and~\ref{sec:proof2}
we provide our two proofs.

	\vspace{2mm}
	
	\noindent {\bf Notations.} We denote by
$\Om := C(\R^+, \R^d)$ the canonical space of continuous
functions from $\R^+$ to $\R^d$,
	which is a Polish space under the locally uniform convergence metric.
	$B$ denotes the canonical process and
$\F = (\Fc_s,\ s \geq 0)$ denotes the canonical filtration.
	The Borel $\sigma$-field of $\Omega$ is denoted by
	$\Fc$ and coincides with $\bigvee_{s \ge 0} \Fc_s$.
	We denote by $\P$ the Wiener measure on 
	$(\Om,\Fc)$ under which the canonical process $B$ is a $\F$-Brownian motion,
	and by $\Nc^{\P}$ the collection of all $\P$--negligible sets of $\Fc$,
	i.e. all sets $A$ included in some $N \in \Fc$ satisfying $\P(N) = 0$.
	We denote by $\F^{\P} = (\Fc^{\P}_s,\  s\ge 0)$ the $\P$--augmented filtration, that is,
	$\Fc^{\P}_s:=\Fc_s\vee\Nc^{\P}$.
We also set $\Fc^{\P} := \Fc \vee \Nc^{\P} = \bigvee_{s \ge 0} \Fc_s^{\P}$.

Finally,
	for all $(t,\om) \in \R^+ \x \Om$, the stopped path of $\om$ at time $t$ is denoted by
	$[\om]_t := (\om(t \wedge s), s\geq 0)$.

\section{A pseudo-Markov property for controlled diffusion processes}
\label{sec:strong_formulation}

\subsection{Controlled stochastic differential equations}
\label{subsec:ctrl_SDE}

	Let $U$ be a Polish space and $S_{d,d}$ be the collection of all square matrices of order~$d$.

Consider two Borel measurable functions
$b: \R^+ \x \Om \x U \to
\R^d$ and $\sigma: \R^+ \x \Om \x U\to S_{d,d}$
	satisfying the following condition:
	for all  $u \in U$,
	the processes $b(\cdot,\cdot,u)$ and $\sigma(\cdot,\cdot,u)$ are $\F$--progressively measurable or, equivalently,
$(t,\xb) \mapsto b(t,\xb,u)$  and
$(t,\xb) \mapsto \sigma(t,\xb,u)$ are $\Bc(\R^+)\ox\Fc$--measurable and		
	\begin{equation*}
		b(t, \xb,u) = b(t, [\xb]_t, u),
		~~
		\sigma(t, \xb,u) = \sigma(t, [\xb]_t, u),
		~~~
		\forall (t,\xb) \in \R^+ \x \Om
	\end{equation*}
(see Proposition~\ref{prop:pre_opt_pro} below).
In addition, we suppose that there exists $C > 0$ such that,
for all $(t, \xb, \yb,u) \in \R^+ \x \Om^2 \x U$,
		\begin{equation} \label{eq:Lip}
\begin{cases}
\displaystyle   \left|b (t,\xb,u)-b(t,\yb,u)\right|
		+
		\left\|\sigma(t,\xb,u)-\sigma(t,\yb,u)\right\|
		\leq
		C \sup_{0 \le s \le t} |\xb(s) - \yb(s)|, \\
\displaystyle   \sup_{{u \in U}} \left(\left|b(t,\xb,u)\right| + \left\|\sigma(t,\xb,u) \right\| \right)
		\leq
		C\left(1+ \sup_{0 \le s \le t} |\xb(s)| \right).
\end{cases}
	\end{equation}
	
	Denote by $\Uc$ the collection of all $U$--valued $\F$--predictable (or progressively measurable, see Proposition~\ref{prop:pre_opt_pro} below) processes.
	Then, given a control $\nu \in \Uc$,
	consider the controlled stochastic differential
    equation (SDE)
	\begin{equation}\label{eq:SDE}
		dX_s ~=~ b(s,X,\nu_s)ds ~+~ \sigma(s, X,\nu_s)dB_s.
	\end{equation}
	As $B$ is still a Brownian motion on $(\Om,\Fc^{\P},\P,\F^{\P})$ (see, e.g., \cite[Thm.2.7.9]{karatzas-shreve}) we may and do consider~(\ref{eq:SDE})
	on this filtered probability space.
		
	Under Condition \eqref{eq:Lip}, for all control $\nu \in \Uc$ and initial condition $(t,\xb) \in \R^+\times \Om$,
	there exists a unique (up to indistinguishability)
	continuous and $\F^{\P}$--progressively measurable process $X^{t,\xb,\nu}$ in $(\Om, \Fc^{\P}, \P)$,
	such that, for all $\theta$ in $\R^+$,
	\begin{equation}\label{eq:SDE_Int}
		X^{t,\xb,\nu}_{\theta}
		~=~
		\xb(t\wedge\theta)
		+
		\int_t^{t\vee\theta} b(s,X^{t,\xb,\nu},\nu_s)ds
		+
		\int_t^{t\vee\theta}\sigma(s, X^{t,\xb,\nu},\nu_s)dB_s,
		\qquad\P-\mathrm{a.s.}
	\end{equation}
	
	Our two proofs of a pseudo-Markov property for $X^{t,\xb,\nu}$ use an identity in law
	which makes it necessary to reformulate the controlled SDE~\eqref{eq:SDE} as a standard SDE.
	
	Given $(t,\nu)\in\R^+\x\Uc$, we define $\bar{b}^{t,\nu}:\R^+\x \Om^2\to\R^{2d}$ and $\bar{\sigma}^{t,\nu}:\R^+\x \Om^2\to S_{2d,d}$ as follows : for all $s\in\R^+$, $\omb=(\om,\om')\in\Om^2$,
	 \begin{align*}
		 \bar{b}^{t,\nu}(s, \omb)
		 ~:=~
		 \begin{pmatrix}
			 0 \\
			 b(s,\om',\nu_s(\om))\indi{s \ge t}
		 \end{pmatrix}
		 ,	~~~
		 \bar{\sigma}^{t,\nu}(s,\omb)
		 ~:=~
		 \begin{pmatrix}
			 \text{Id}_d \\
			 \sigma(s, \om', \nu_s(\om))\indi{s \ge t}
		 \end{pmatrix}.
	 \end{align*}
	Then, given $\xb$ in $\Om$, $Y^{t,\xb,\nu}:=(B,X^{t,\xb,\nu})$ is the unique (up to indistinguishability)
	continuous and $\F^{\P}$--progressively measurable process in $(\Om, \Fc^{\P}, \P)$
	such that, for all $\theta$ in $\R^+$,
	 \begin{equation}\label{eq:SDE_Int_bis}
		 Y^{t,\xb,\nu}_{\theta}
		 ~=~
		 \begin{pmatrix}
			0\\
			\xb(t\wedge\theta)\\
		 \end{pmatrix}
		 +
		 \int_0^{\theta} \bar{b}^{t,\nu}(s,Y^{t,\xb,\nu})ds
		 +
		 \int_0^{\theta} \bar{\sigma}^{t,\nu}(s,Y^{t,\xb,\nu})dB_s,
		 \quad\P\mbox{--a.s.}
	 \end{equation}

Define
pathwise uniqueness and uniqueness in law
for standard SDEs as, respectively, in
Rogers and Williams~\cite[Def.V.9.4]{rogers-williams-II}
and \cite[Def.V.16.3]{rogers-williams-II}).
In view of the celebrated Yamada and Wanabe's Theorem
(see, e.g., \cite[Thm.V.17.1]{rogers-williams-II}),
the former implies the latter and the SDE~\eqref{eq:SDE_Int_bis} satisfies uniqueness in law.

	\begin{remark} \label{rem:formulation}
		\rmi The stochastic integrals and the solution to the SDE \eqref{eq:SDE} are defined as continuous and adapted to the augmented filtration $\F^{\P}$.

\noindent \rmii In an abstract probability space equipped
with a standard Brownian motion $W$, our formulation is
equivalent to
choosing controls as predictable processes w.r.t. the filtration
generated by $W$. Indeed such a process  can always be
represented as
$\nu(W_{\cdot})$ for some $\nu$ in $\Uc$ (see
Proposition~\ref{prop:abstract_space} below). This fact
plays a crucial role in our justification of the pseudo-Markov
property for controlled diffusions.
In their analyses of stochastic control problems,
Krylov~\cite{krylov} or Fleming and
Soner~\cite{fleming-soner} do not need to use
this property and deal with controls adapted to
an arbitrary filtration.

		\noindent \rmiii
		We could have defined controls as $\F^{\P}$--predictable processes since any $\F^{\P}$--predictable process is indistinguishable from an $\F$--predictable process
		(see Dellacherie and Meyer~\cite[Thm.IV.78 and Rk.IV.74] {dellacherie-meyer}). Notice also that the notions of predictable, optional and progressively measurable process w.r.t. $\F$ coincide (see Proposition~\ref{prop:pre_opt_pro} below).
		
	\end{remark}

\subsection{A pseudo-Markov property for controlled diffusion processes}

	Before formulating our main result, we introduce
the class of the shifted control processes constructed
by concatenation
	of paths: for all $\nu$ in $\Uc$ and $(t,\w)$ in $\R^+ \x \Om$ we set
	\begin{equation*}
		\nu_s^{t,\w}(\om) := \nu_s(\w \ox_t \om), \quad \forall (s,\om) \in \R^+ \x \Om,
	\end{equation*}
	where $\w \ox_t \om$ is the function in ${\Om}$ defined by
	\begin{equation*}
		(\w \ox_t \om)(s) ~:=~
		\begin{cases}
			\w(s), & \text{if } 0\leq s\leq t, \\
			\w(t)  +\om(s) - \om(t), & \text{if } s\geq t.
		\end{cases}
	\end{equation*}

	Our main result is the following
   pseudo-Markov property
	for controlled diffusion processes. As already
mentioned, it is most often considered as classical or obvious;
however, to the best of our knowledge, its precise statement
and full proof are original.

	\begin{theorem}\label{theo:main}
	Let $\Phi : \Om \to \R^+$ be a bounded Borel measurable function
	and let $J$ be defined as
	\begin{equation*}
 		J \big(t, \xb,\nu \big)
		~:=~
		\E^{\P} \big[ \Phi \big(X^{t,\xb,\nu }\big)\big],
		~~~
		\forall (t,\xb, \nu) \in \R^+ \x \Om \x \Uc.
	\end{equation*}
Under Condition~(\ref{eq:Lip}),
		for all $(t,\xb,\nu) \in \R^+\times \Om \times\Uc$
		and $\F^{\P}$--stopping time $\tau$ taking values in $[t,+\infty)$ we have
		 \begin{equation} \label{eq:conditionning}
			 \E^{\P}\left[ \Phi \left(X^{t,\xb,\nu}\right)\Big|\,\Fc^{\P}_\tau\right](\om)
			 ~=~
			 J\left(\tau(\om), \left[X^{t,\xb,\nu}\right]_{\tau}(\om), \nu^{\tau(\om),\om}\right),
			 \quad
			 \P(d\om)-\mathrm{a.s.}
		 \end{equation}
	\end{theorem}

	\begin{remark}\label{rem:strong_cond}
		\rmi Our formulation slightly differs from Fleming and Souganidis~\cite{fleming-souganidis}
		who consider deterministic times $\tau=s$
		and conditional expectations given the
		non-augmented $\sigma$--field $\Fc_s$.

		\noindent \rmii We work with the augmented $\sigma$--fields $\F^{\P}$  in order to make
		the solutions $X^{t,\xb,\nu}$ adapted.

\noindent \rmiii One motivation
to consider stopping times $\tau$ rather than deterministic
times is that, to study viscosity solutions
	        to Hamilton-Jacobi-Bellman equations, one
often uses  first exit times of $X^{t,\xb, \nu}$ from Borel
subsets of $\R^d$.

		\noindent \rmiv It is not clear to us whether the function $J$ is measurable w.r.t. all its arguments.
		However Theorem~\ref{theo:main} implies that
the r.h.s. of~\eqref{eq:conditionning} is a $\Fc^{\P}_{\tau}$--measurable random variable.
	\end{remark}

	\noindent
	In Section \ref{subsec:discussion}
	we discuss technical subtleties hidden in
Equality~(\ref{eq:conditionning}) and
	point out some of the difficulties which motivate us to
propose a detailed proof. Before further discussions,
	we show how the pseudo-Markov property is used to prove parts of the DPP.
	
        Consider the value function
	\begin{equation} \label{eq:def_V}
		V(t,\xb)
		~:=~
		\sup_{\nu \in \Uc} \E^{\P}\left[ \Phi( X^{t,\xb,\nu}) \right],
		~~~
		\forall (t,\xb) \in \R^+ \x \Om.
	\end{equation}

	\begin{proposition} \label{prop:dpp_le}
		\rmi	For all $(t,\xb)$ in $\R^+ \x \Om$, it holds that
		\begin{equation} \label{eq:value_fct_equiv}
					V(t,\xb) ~=~ \sup_{\mu \in \Uc^t} \E^{\P}\left[ \Phi(X^{t,\xb,\mu}) \right],
		\end{equation}
		where $\Uc^t$ denotes the set of all $\nu \in \Uc$ which are independent of $\Fc_t$ under $\P$.
		
		\vspace{1mm}
		
		\noindent \rmii Suppose in addition that the value function $V$ is measurable.
		Then, for all $(t,\xb)\in\R^+ \x \Om$, $\varepsilon>0$, there exists
		$\nu\in\Uc$ such that for all $\F^{\P}$--stopping times $\tau$ taking values in $[t,\infty)$, one has
		\begin{equation} \label{eq:dpp_le}
			V(t,\xb) - \varepsilon ~\leq~ \E^{\P} \left[ V\left(\tau, \left[X^{t,\xb,\nu}\right]_{\tau}\right) \right].
		\end{equation}
	\end{proposition}

	\proof \rmi For \eqref{eq:value_fct_equiv}, it is enough to prove that
	\begin{equation*}
			V(t,\xb) ~\leq~ \sup_{\mu \in \Uc^t} \E^{\P}\left[ \Phi(X^{t,\xb,\mu}) \right],
	\end{equation*}
	since the other inequality results from $\Uc^t\subset \Uc$.
	Now, let $\nu$ be an arbitrary control in $\Uc$.
	Apply Theorem \ref{theo:main} with $\tau \equiv t$. It comes:
		\begin{equation*}
		J \big(t, \xb,\nu \big) ~=~
\int_{\Om} {J\left(t,[\xb]_t,\nu^{t,\om}\right)} \,\P(d\om)
~=~\int_{\Om} {J\left(t,\xb,\nu^{t,\om}\right)} \,\P(d\om).
	\end{equation*}
	Then \eqref{eq:value_fct_equiv} follows from the fact that, for all fixed $\om\in\Om$,
	the control $\nu^{t,\om}$ belongs to $\Uc^t$.
	
	\vspace{1mm}
	
	\noindent \rmii Notice that, for all $\varepsilon > 0$,
	one can choose an $\varepsilon$--optimal control $\nu$, that is,
	$$ 	V(t,\xb) - \varepsilon ~\leq~ J(t, \xb, \nu). $$
	Apply Theorem \ref{theo:main} with this control $\nu$. It comes:
	\begin{equation*}
		V(t,\xb) - \varepsilon
		~\leq~
		\int_{\Om} {J\left(\tau(\om), \left[X^{t,\xb,\nu}\right]_{\tau}(\om), \nu^{\tau(\om), \om}\right)} \,\P(d\om)
		~\le~
		\E^{\P} \Big[ V\left(\tau, \left[X^{t,\xb,\nu}\right]_{\tau}\right) \Big].	
	\end{equation*}

	\qed

	\begin{remark}
		Inequality~\eqref{eq:dpp_le} is the `easy' part of the DPP.
		Equality~\eqref{eq:value_fct_equiv}, combined with the continuity of the value function,
		is a key step in classical proofs of the `difficult' part of the DPP, that is:
		for all control $\nu$ in $\Uc$ and all $\F^{\P}$--stopping time $\tau$ taking values in $[t,\infty)$, one has
		\begin{equation*}
			V(t,\xb) ~\geq~ \E^{\P} \Big[ V\left(\tau, \left[X^{t,\xb,\nu}\right]_{\tau}\right) \Big].
		\end{equation*}
	\end{remark}

\subsection{Discussion on Theorem \ref{theo:main}}
\label{subsec:discussion}

\paragraph{On the hypothesis on $b$ and $\sigma$.}

    In Theorem \ref{theo:main} the coefficients $b$ and $\sigma$ are supposed
    to satisfy Condition~\eqref{eq:Lip} and thus the solutions to the SDE~\eqref{eq:SDE_Int_bis} are pathwise unique.
However,  we only need uniqueness in law in our
second proof of Theorem~\ref{theo:main} which is based
on controlled martingale problems related to the
SDE~\eqref{eq:SDE_Int_bis}. Hence Condition~\eqref{eq:Lip}
can be relaxed to weaker conditions which imply weak
uniqueness. However we have not followed this way here to
avoid too heavy notations and to remain within the classical
family of controlled SDEs with pathwise unique solutions.

\paragraph{Intuitive meaning and measurability issues.}

	The intuitive meaning of Theorem \ref{theo:main} is as follows.
	Notice that $X^{t,\xb,\nu}$ satisfies: for all $\theta\in\R^+$,
	\begin{equation} \label{eq:faux_flow}
		X^{t,\xb, \nu}_{\theta}
		~=~
		X^{t,\xb,\nu}_{\tau \wedge \theta}
		+ \int_{\tau}^{\tau \vee \theta} b(s, X^{t,\xb,\nu}, \nu_s) ds
		+ \int_{\tau}^{\tau \vee \theta} \sigma(s, X^{t,\xb,\nu}, \nu_s) dB_s,\quad \P-\mathrm{a.s.}
	\end{equation}
	If a regular conditional probability~(r.c.p.)
	$(\P_{\w}, \w \in \Om)$ of $\P$ given $\Fc^{\P}_{\tau}$
	were to exist (see, e.g.,  Karatzas and
 Shreve~\cite[Sec.5.3.C]{karatzas-shreve}),
 then Equation~\eqref{eq:faux_flow} would be satisfied
	$\P_{\w}$--almost surely and, in addition, we would have
	\begin{equation}\label{eq:faux_prop}
	 \P_{\w} \Big( \tau = \tau(\w), \left[X^{t,\xb,\nu}\right]_{\tau} = \left[X^{t,\xb,\nu}\right]_{\tau}(\w), \nu = \nu^{\tau(\w),\w} \Big) = 1.
	\end{equation}
	Hence, Theorem \ref{theo:main} would follow from the uniqueness in law of solutions to
	Equation~\eqref{eq:SDE_Int_bis}.
Unfortunately the situation is not so simple for the following
reasons.
	First,
	since $\Fc^{\P}_{\tau}$ is complete,
	a r.c.p. of $\P$ given $\Fc^{\P}_{\tau}$ does not exist (see~\cite[Thm.10]{faden}).
	Second, even if $\tau$ were a $\F$--stopping time and $(\P_{\w}, \w \in \Om)$ were
	a r.c.p. of $\P$ given $\Fc_{\tau}$, then
	$\P_{\w}$ would be defined as a measure on $\Fc$
	whereas $X^{t,\xb,\nu}$ is adapted to the $\P$--augmented filtration $\F^{\P}$: hence,
Equality~\eqref{eq:faux_prop} needs to be
	rigorously justified.
	Third, one needs to check that the stochastic integral in \eqref{eq:faux_flow},
	which is constructed under $\P$, agrees with 
	the stochastic integral constructed under $\P_{\w}$ for $\P$-a.a. $\w$.


\section{Technical Lemmas}
\label{sec:technical_lemmas}

	In order to solve the measurability issues mentioned above, we establish three technical lemmas which, to the best of our knowledge, are not available in the literature.

The first lemma improves the classical Dynkin theorem
which states that, given an arbitrary probability space
and arbitrary filtration $(\mathcal{H}_t)$,
a stopping time w.r.t. to the augmented
filtration of $(\mathcal{H}_t)$
is a.s. equal to a  stopping time w.r.t. $(\mathcal{H}_{t+})$:
see, e.g., \cite[Thm.II.75.3]{rogers-williams-I}.
The improvement here is allowed by the fact that we are dealing
		with the augmented Brownian filtration 
		(more generally, we could deal with the natural
		augmented filtration generated by a Hunt process
		with continuous paths: see Chung and Williams~\cite[Sec.2.3,p.30-32]{chung-williams}).

	\begin{lemma} \label{lemma:tau_eta}
		Let $\tau$ be a $\F^{\P}$--stopping time. There exists a $\F$--stopping time $\eta$ such that
		\begin{equation} \label{eq:equiv_stoptime}
			\P \big( \tau = \eta \big) = 1
			~~\mbox{and}~~
			\Fc^{\P}_{\tau} ~=~ \Fc_{\eta}  ~\vee~ \Nc^{\P}.
		\end{equation}
	\end{lemma}

	\proof
\textsl{First step.}
	As $\F^{\P}$ is the augmented Brownian filtration,
	all $\F^{\P}$--optional process is predictable and hence
	all $\F^{\P}$--stopping time is predictable (see, e.g., Revuz and
Yor \cite[Cor.V.3.3]{revuz-yor}).
	It follows from Dellacherie and Meyer \cite[Thm.IV.78]{dellacherie-meyer} that
	there exists a $\F$--stopping time (actually, a predictable time) $\eta$ such that
	\begin{equation*}
		\P( \tau = \eta) ~=~ 1.
	\end{equation*}

\noindent
\textsl{Second step.}
We now prove $\Fc_{\eta} \vee \Nc^{\P} \subseteq \Fc^{\P}_{\tau}$.
	First, since $\Fc^{\P}_0 \subset \Fc^{\P}_{\tau}$,  we
have $\Nc^{\P} \subset \Fc^{\P}_{\tau}$.
	Second, $\eta$ is a $\F^{\P}$--stopping time.
	Since $\tau = \eta, ~\P-\mathrm{a.s.}$ and $\Fc^{\P}_{\tau}$, $\Fc^{\P}_{\eta}$ are both $\P$--complete,
	one has $\Fc^{\P}_{\tau} = \Fc^{\P}_{\eta}$, from which
	$\Fc_{\eta} \subset \Fc^{\P}_{\eta} = \Fc^{\P}_{\tau}$.

	\vspace{1mm}

	\noindent
\textsl{Last step.}
It therefore only remains to prove $\Fc^{\P}_{\tau} \subset \Fc_{\eta} \vee \Nc^{\P}$.
	Let $A \in \Fc^{\P}_{\tau}$. In view of
~\cite[Thm.IV.64]{dellacherie-meyer}, there exists a $\F^{\P}$--optional (or equivalently $\F^{\P}$--predictable) process $X$ such that $\indi{A} = X_{\tau}$.
	By using Theorem IV.78 and Remark IV.74 in~ \cite{dellacherie-meyer},
	there exists a $\F$--predictable process $Y$ which is indistinguishable from $X$.
	It follows that $\indi{A} = X_{\tau} = Y_{\tau} = Y_{\eta}, ~\P-\mathrm{a.s.}$,
	which implies that $A \in \Fc_{\eta} \vee \Nc^{\P}$ since $Y_{\eta}$ is $\Fc_{\eta}$--measurable.
	This ends the proof. \qed
	
	\vspace{1mm}
	
	The two next lemmas are used in Section~\ref{sec:proof1} only. They concern the r.c.p. $(\P_{\w})_{\w\in\Om}$ of $\P$ given a  $\sigma$--algebra $\Gc\subset\Fc$. They allow us to circumvent the difficulties raised by the fact that $\P_{\w}$ is not a measure on $\Fc^{\P}$. The key idea is to notice that, if $N$ in $\Fc$ satisfies $\P(N)=0$, then there is a $\P$--null set $M\in\Fc$ (depending on $N$) such that $\P_{\w}(N)=0$ for all $\w\in\Om\setminus M$. 
	Hence any subset of $N$ belongs to $\Nc^{\P}$ and, for all $\w\in\Om\setminus M$, to the set $\Nc^{\P_{\w}}$ of all $\P_{\w}$--negligible sets of $\Fc$.

 \begin{lemma} \label{lemm:rcpd}
	Define $\Gc^{\P}:=\Gc\vee\Nc^{\P}$ and $\Gc^{\P_{\w}}:=\Gc\vee\Nc^{\P_{\w}}$. Let $E$ be a Polish space and $\xi:\Om\to E$ be a $\Gc^{\P}$--random variable. Then \\
	\rmi There exists a $\Gc$--random variable $\xi^0:\Om\to E$ such that
	\begin{equation*}
	 \P\left(\xi=\xi^0\right)=1.
	\end{equation*}
	\rmii  For $\P$--a.a. $\w\in\Om$, $\xi$ is a $\Gc^{\P_{\w}}$--random variable and
	\begin{equation*}
		\P_{\w}\left(\xi=\xi(\w)\right)=1.
	\end{equation*}
	\rmiii Let $\Fc^{\P_{\w}}:=\Fc\vee\Nc^{\P_{\w}}$ and $\zeta$ be a $\P$--integrable $\Fc^{\P}$--random
	variable. Then for $\P$--a.a. $\w\in\Om$, $\zeta$ is $\Fc^{\P_{\w}}$--measurable and
	$$ \Ee[ \zeta~|~ \Gc^{\P} ] (\w) = \int_{\Om} \zeta(\w ')
	\P_\w(d\w '). $$
	 \end{lemma}
 \begin{proof}
	\rmi As $E$ is Polish, there exists a Borel isomorphism between $E$ and a subset of $\R$. This allows us
  to only consider real-valued random variables.
	A monotone class argument allows us to deal with
random variables of the type $\indi{G}$ with $G$ in
$\Gc^{\P}$, from which the result $(i)$  follows since
	\begin{equation} \label{prop:GP}
		\Gc^{\P} ~=~ \left\{G\in\Fc^{\P};\ \exists\, G^0\in\Gc,\ G\triangle G^0\in\Nc^{\P}\right\},
	\end{equation}
	where $\triangle$ denotes the symmetric difference (see, e.g., Rogers and Williams \cite[Ex.V.79.67a]{rogers-williams-I}).\\
	\rmii Let $\xi^0$ be a $\Gc$--random variable such that $\P(\xi=\xi^0)=1$. 
 In other words, there is a $\P$--null set $N\in\Fc$ 
 such that $\xi^0(\om) = \xi(\om)$ for all $\om \in \Om \setminus N$. 
 Hence, for $\P$--a.a. $\w$, we have $\P_{\w}(N) = 0$ 
 and $\{\xi\in A\}\triangle\{\xi^0\in A\} \subset N$ 
 belongs to $\Nc^{\P_{\w}}$ for all Borel set $A$.  
 It results from~(\ref{prop:GP}) with $\P_{\w}$ in place of $\P$ that $\xi$ is $\Gc^{\P_{\w}}$--measurable.
 Moreover, since $\xi^0$ is an $\Gc$--random variable taking values in a Polish space, we have
	\begin{equation*}
		\P_{\w}\left(\xi^0=\xi^0(\w)\right)=1,
	\end{equation*}
	for $\P$--a.a. $\w\in\Om$. Then, for all $\w\in\Om\setminus N$ such that
	\begin{equation*}
		\P_{\w}(N) = 0 ~~and~~ \P_{\w}\left(\xi^0=\xi^0(\w)\right)=1,
	\end{equation*}
	we have
	\begin{equation*}
		\P_{\w}\left(\xi=\xi(\w)\right)=1.
	\end{equation*}

\noindent \rmiii
Using the same arguments as in the proof of (ii), we get that $\zeta$ is $\Fc^{\P_\w}$--measurable.
Now, let $\chi$ be a bounded $\Gc^{\P}$--measurable random variable. 
In view  of~(i), there exists a $\Gc$--measurable (resp. $\Fc$--measurable) $\chi_0$ (resp. $\zeta_0$) random variable
such that $\chi=\chi_0$ (resp. $\zeta=\zeta_0$) $\P$--a.s. Therefore,
\begin{equation*}
 \E^{\P}\left[\zeta \chi\right] = \E^{\P}\left[\zeta_0 \chi_0\right] = \E^{\P}\left[\E^{\P}[ \zeta_0~|~ \Gc ] \chi\right].
\end{equation*}
Hence it holds
\begin{equation*}
 \E^{\P}[ \zeta~|~ \Gc^{\P} ] (\w) = \int_{\Om} \zeta_0(\w ')
	\P_\w(d\w '),\qquad \P(d\w)-a.s.
\end{equation*}
We then conclude by using that $\zeta$ is $\Fc^{\P_{\w}}$--measurable and $\P_{\w}(\zeta=\zeta_0)=1$ for $\P$--a.a $\w\in\Om$.

 \end{proof}

 \begin{lemma}\label{lemm:predictable}
   Let $\F^{\P_{\w}}$ be the $\P_{\w}$--augmented
filtration of $\F$. Let $E$ be a Polish space and $Y:\R^+\x\Om\to E$ be a $\F^{\P}$--predictable process. Then, for $\P$--a.a. $\w\in\Om$, $Y$ is $\F^{\P_{\w}}$--predictable.
 \end{lemma}

 \begin{proof}
  As already noticed in the proof of Lemma~\ref{lemma:tau_eta}, there exists
a $\F$--predictable process indistinguishable from $Y$
under $\P$.
The result thus follows from our first arguments in the proof of Lemma~\ref{lemm:rcpd}~(ii).
 \end{proof}

	We end this section by two propositions which will not be used in the proof of the main results but were implicitly used
in the formulation of stochastic control problems in Section~\ref{subsec:ctrl_SDE}.

	The proposition below ensures that classical notions of measurability coincide for processes defined on the canonical space $\Om$. This is no longer true when $\Om$ is the space of c\`adl\`ag functions. For a precise statement, see Dellacherie and Meyer~\cite[Thm.IV.97]{dellacherie-meyer}.

	\begin{proposition} \label{prop:pre_opt_pro}
		Let $E$ be a Polish space and
		$Y : \R^+ \x \Om \to E$ be a stochastic process.
		Then the following statements are equivalent : \\
		\rmi $Y$ is $\F$--predictable.\\
		\rmii $Y$ is $\F$--optional. \\
		\rmiii $Y$ is $\F$--progressively measurable.\\
		\rmiv $Y$ is $\Bc(\R^+) \ox \Fc$--measurable and $\F$--adapted. \\
		\rmv $Y$ is $\Bc(\R^+) \ox \Fc$--measurable and satisfies
		\begin{equation*}
		 Y_s(\om)~=~Y_s([\om]_s),\quad \forall (s,\om)\in\R^+\x\Om.
		\end{equation*}
	\end{proposition}

	\proof It is well known that $(i)\Rightarrow(ii)\Rightarrow(iii)\Rightarrow(iv)$. Now we show that $(iv)\Rightarrow(v)$.
		Remember that $\Fc_s=\sigma([\cdot]_s:\om\mapsto[\om]_s)$ for all $s\in\R^+$. Since $Y_s$ is $\Fc_s$--measurable and $E$ is Polish,
		Doob's functional representation Theorem (see, e.g.,  Kallenberg~\cite[Lem.1.13]{kallenberg}) implies that
		there exists a random variable $Z_s$ such that $Y_s(\om)=Z_s([\om]_s)$ for all $\om\in\Om$. By noticing that $[\cdot]_s\circ [\cdot]_s = [\cdot]_s$, we deduce that
		\begin{equation*}
		 Y_s([\om]_s) = Z_s([\om]_s) = Y_s(\om),\quad \forall (s,\om)\in\R^+\x\Om.
		\end{equation*}
		To conclude it remains to prove that $(v)\Rightarrow(i)$. Clearly it is enough to show that $[\cdot]:(s,\om)\mapsto[\om]_s$ is a predictable
		process. In other words, we have to show that for any $C\in\Fc$, $[\cdot]^{-1}(C)$ is a predictable event. When
		$C$ is of the type $B_{t}^{-1}(A)$ with $t\in\R^+$ and $A\in\Bc(\R^d)$, the result holds true since $B_t\circ [\cdot]:(s,\om)\mapsto \om(t\wedge s)$ is predictable as a continuous and adapted $\R^d$--valued process. We conclude  by observing that the former events generate the $\sigma$--algebra $\Fc$.
	\qed
	
	\vspace{1mm}
	
	Now, let $(\Om^*, \Fc^*)$ be an abstract measurable space equipped with a measurable, continuous process $X^* = (X^*_t)_{t \ge 0}$.
	Denote by $\F^* = (\Fc^*_t)_{t \ge 0} $ the filtration generated by $X^*$,
	i.e. $\Fc^*_t = \sigma \big\{ X^*_s, ~s \le t \big\}$.
	
	\begin{proposition} \label{prop:abstract_space}
		Let $E$ be a Polish space.
		Then a process $Y: \R^+ \x \Om^* \to E$ is $\F^*$-predictable
		if and only if there exists a process $\phi: \R^+ \x \Om \to E$ defined on the canonical space $\Om$ and
		$\F$--predictable such that
		\begin{equation} \label{eq:Yt}
			Y_t(\om^*) ~=~ \phi \big(t, X^*(\om^*) \big)
			~=~ \phi \big(t, \left[X^*\right]_{t}(\om^*) \big),
			~~~~\forall (t, \om^*) \in \R^+ \x \Om^*.
		\end{equation}
	\end{proposition}
	
	\begin{proof}
		\textsl{First step.}
Let $\phi: \R^+ \x \Om \to E$ be a $\F$--predictable process on the canonical space.
		Notice that $(t, \om^*) \mapsto t$ and $(t,\om^*) \mapsto [X^*]_{t}(\om^*)$ are both $\F^*$--predictable.
		It follows that $(t, \om^*) \mapsto Y_t(\om^*) := \phi(t, [X^*]_{t}(\om^*))$ is also $\F^*$--predictable.
		
		\vspace{1mm}
		
		\noindent \textsl{Second step.}
We now prove the converse relation.
		Suppose that $Y$ is a $\F^*$--predictable process of the type  $Y_s(\om^*) = \indi{(t_1, t_2] \x A}(s, \om^*)$, where $0 < t_1 < t_2$ and $A \in \Fc^*_{t_1}$.
	   There exists $C$ in $\Fc$ such that $A = ([X^*]_{t_1})^{-1}(C)$, so that \eqref{eq:Yt} holds true with
		\begin{equation*}
			\phi(s, \w) := \indi{(t_1, t_2] \x C}(s, [\w]_{t_1}).
		\end{equation*}
		In view of Proposition~\ref{prop:pre_opt_pro}, it is clear that the above $\phi$ is $\F$--predictable.
		The same arguments show that~\eqref{eq:Yt} holds also true if $Y$ is of the form $Y_s(\om^*) = \indi{\{0\} \x A}$ with $A \in \Fc_0^*$.
		Notice that the predictable $\sigma$-field on $\R^+ \x \Om^*$ w.r.t. $\F^*$ is generated by the collection of all sets of the form $\{0\} \x A$ with $A \in \Fc^*_0$ and of the form $(t_1, t_2] \x A$ with $0 < t_1 < t_2$ and $A \in \Fc^*_{t_1}$ (see, e.g.,
        \cite[Thm.IV.64]{dellacherie-meyer}).
		It then remains to use the monotone class theorem
to conclude.
	\end{proof}

\section{A first proof of Theorem \ref{theo:main}}
\label{sec:proof1}

In this section we develop the arguments sketched by
Fleming and Souganidis \cite{fleming-souganidis} in the proof
of  their Lemma 1.11.

	Let $\tau$ be a $\F^{\P}$--stopping time taking values in $[t,+\infty)$.
We have
	\begin{equation*}
		X^{t,\xb,\nu}_{\theta}
		=
		X^{t,\xb,\nu}_{\tau\wedge\theta}
		+
		\int_{\tau}^{\tau\vee\theta} b(s,X^{t,\xb,\nu},\nu_s)ds
		+
		\int_{\tau}^{\tau\vee\theta}\sigma(s, X^{t,\xb,\nu},\nu_s)dB_s,
		~\forall \theta\geq 0, ~\P\mbox{--a.s.}
	\end{equation*}

Now, it follows from Lemma~\ref{lemma:tau_eta} that
	there is a $\F$--stopping time $\eta$ such that \eqref{eq:equiv_stoptime} holds true.
	Let $(\P_{\w})_{\w \in \Om}$ be a family of r.c.p. of $\P$ given $\Fc_{\eta}$. In view of Lemma~\ref{lemm:rcpd},
   for $\P$--a.a. $\w \in \Om$,

	\begin{equation}\label{eq:condequiv}
		\E^{\P}\left[ \Phi \left(X^{t,\xb,\nu}\right)\Big|\,\Fc^{\P}_\tau\right](\w)
		~=~
		\E^{\P_{\w}}\left[ \Phi \left(X^{t,\xb,\nu}\right)\right],
	\end{equation}
	and
	\begin{equation}\label{eq:rcpdprop}
		\P_{\w} \big( \tau = \tau(\w),
		        ~[X^{t,\xb,\nu}]_{\tau} = [X^{t,\xb,\nu}]_{\tau}(\w),
			~\nu = \nu^{\tau(\w),\w}
		\big)
		~=~
		1.
	\end{equation}

 In view of Lemma
\ref{lemm:integral_diff_P} below we have, for $\P$--a.a. $\w \in \Om$,
	 \begin{equation*}
		\int_{[\tau,\theta]}^{\P} {\sigma\left(s, X^{t,\xb,\nu}, \nu_s\right)} d B_s
		~=~
		\int_{[\tau(\w),\theta]}^{\P_{\w}} {\sigma\left(s, X^{t,\xb,\nu}, \nu_s\right)} d B_s,\quad\forall\, \theta\geq\tau(\w),\ \P_{\w}-\mathrm{a.s.},
	\end{equation*}
	where the l.h.s. (resp. r.h.s.) term denotes the stochastic integral constructed under $\P$ (resp. $\P_{\w}$).
	It follows from~\eqref{eq:rcpdprop} that, for $\P$--a.a. $\w\in\Om$,
	\begin{eqnarray*}
		X^{t,\xb,\nu}_{\theta}
		&=&
		X^{t,\xb,\nu}_{\tau\wedge\theta}(\w)
		~+~
		\int_{\tau(\w)}^{\tau(\w)\vee\theta} b(s,X^{t,\xb,\nu},\nu^{\tau(\w),\w}_s)ds \\
		&& +~
		\int_{\tau(\w)}^{\tau(\w)\vee\theta}\sigma(s, X^{t,\xb,\nu},\nu^{\tau(\w),\w}_s)dB_s,
		~~~~
		\forall \theta\geq 0,\ \P_{\w}-\mathrm{a.s.}
	\end{eqnarray*}
	
	Notice that, by Lemma~\ref{lemm:predictable}, $X^{t,\xb,\nu}$ is $\F^{\P_{\w}}$--adapted, for $\P$--a.a. $\w\in\Om$.
Hence, $X^{t,\xb,\nu}$ is the solution of SDE~\eqref{eq:SDE}
with initial condition $(\tau(\w),[X^{t,\xb,\nu}]_{\tau}(\w))$
	and  control $\nu^{\tau(\w),\w}$ in $(\Om,\F^{\P_{\w}},\P_{\w})$ for $\P$--a.a. $\w\in\Om$.
As the SDE~\eqref{eq:SDE_Int_bis} satisfies uniqueness in law, 
	 the law of $X^{t,\xb,\nu}$ under $\P_{\w}$
	coincides with the law of $X^{\tau(\w),[X^{t,\xb,\nu}]_{\tau}(\w),\nu^{\tau(\w),\w}}$ under $\P$.
	We then conclude the proof by using~\eqref{eq:condequiv}.
	\qed
	
	 \begin{lemma} \label{lemm:integral_diff_P}
		Let $H$ be a $\F^{\P}$--predictable process such that
		\begin{equation*}
			\int_0^\theta {(H_s)^2} ds < +\infty, \quad \forall \theta\geq 0,\ \P-\text{a.s.}
		\end{equation*}
		Then, using the same notation as above for the stochastic
integrals, we have: For $\P$--a.a. $\w\in\Om$,
		$\int_{[\tau,\theta]}^{\P} {H_s} d B_s$ is $\Fc^{\P_{\w}}$--measurable and
		 \begin{equation}\label{eq:intsto}
		 	\int_{[\tau,\theta]}^{\P} {H_s} d B_s ~=~ \int_{[\tau(\w),\theta]}^{\P_{\w}} {H_s} d B_s,\quad \forall\,\theta\geq \tau(\w),\ \P_{\w}-\text{a.s.}
	 	\end{equation}
	\end{lemma}
	
	\begin{proof}
	In this proof we implicitly use Lemma~\ref{lemm:predictable} to consider $\F^{\P}$--predictable processes, such as the stochastic integrals defined under $\P$, as $\F^{\P_{\w}}$--predictable processes for $\P$--a.a. $\w\in\Om$.
	In particular, the l.h.s. of \eqref{eq:intsto} is $\Fc^{\P_{\w}}$--measurable for $\P$--a.a. $\w\in\Om$.

A standard localizing procedure allows us to assume that
	 \begin{equation*}
		\E^{\P}\left[\int_0^{+\infty} {(H_s)^2} ds\right] < +\infty.
	 \end{equation*}

	 Now let $(H^{(n)})_{n\in\mathbb{N}}$  be a sequence of simple processes such that
	 \begin{equation}\label{eq:simple-approx}
		\lim_{n\to\infty} {\E^{\P}\left[\int_{0}^{+\infty} {\left(H^{(n)}_s - H_s\right)^2} ds\right]} = 0.
	 \end{equation}

We then write
\begin{equation*}
\begin{split}
\int^{\P} H_s dB_s  - \int^{\P_{\w}} H_s dB_s
&= \int^{\P} H^{(n)}_s dB_s  - \int^{\P_{\w}} H^{(n)}_s dB_s
\\
&\quad + \int^{\P} (H_s-H^{(n)}_s) dB_s
- \int^{\P_{\w}} (H_s-H^{(n)}_s) dB_s
\end{split}
\end{equation*}
and notice that the first difference is null since the stochastic integral is defined pathwise when the integrand is a simple process.

By taking conditional expectations in~\eqref{eq:simple-approx}, we get that there exists
 a subsequence such that
	 \begin{equation*}
		\lim_{n\to\infty} {\E^{\P_{\w}}\left[\int_{0}^{+\infty} {\left(H^{(n)}_s - H_s\right)^2} ds\right]} = 0,
	 \end{equation*}
	 for $\P$--a.a. $\w\in\Om$. Hence, by Doob's inequality and It\^{o}'s isometry,
	 \begin{equation*}
	  \lim_{n\to\infty} {\E^{\P_{\w}} \left[\sup_{\theta\geq \tau(\w)}{\left\{\left(\int_{[\tau(\w),\theta]}^{\P_{\w}} {H^{(n)}_s} dB_s - \int_{[\tau(\w),\theta]}^{\P_{\w}} {H_s} dB_s\right)^2\right\}}\right]} = 0.
	 \end{equation*}
	 To conclude, it thus suffices to prove that
	 \begin{equation*}
	  \lim_{n\to\infty} \E^{\P_{\w}}
\left[ \sup_{\theta\geq \tau(\w)}
\left\{ \left(\int_{[\tau(\w),\theta]}^{\P}
{H^{(n)}_s} dB_s - \int_{[\tau,\theta]}^{\P} {H_s}
dB_s\right)^2\right\} \right] = 0.
	 \end{equation*}
	 for $\P$--a.a. $\w\in\Om$.
We cannot use Doob's inequality and It\^{o}'s isometry without care because the stochastic integrals are
built under $\P$ and the expectation is computed under
$\P_{\w}$. However we have
	 \begin{equation*}
	  \lim_{n\to\infty} {\E^{\P} \left[\sup_{\theta\geq0}{\left\{\left(\int_{[0,\theta]}^{\P} {H^{(n)}_s \indi{s\geq \tau}} d B_s - \int_{[0,\theta]}^{\P} {H_s \indi{s\geq \tau}} d B_s\right)^2\right\}}\right]} = 0.
	 \end{equation*}
	 Thus we can proceed as above by taking  conditional
expectation  and extracting a new subsequence. That ends the
proof.
	\end{proof}
	
	\begin{remark}
		   To avoid the technicalities of this first proof
		   	of Theorem~\ref{theo:main},
		   	a natural attempt consists in solving
		   	the equations~\eqref{eq:SDE} in the space $\Om$ equipped
		   	with the non augmented filtration. This may be achieved by following Stroock and Varadhan's approach~\cite{stroock-varadhan}
		   	to stochastic integration. This way leads to
		   	right-continuous and only $\P$--a.s continuous stochastic integrals and solutions.
		   	As they are not valued in a Polish space, new delicate technical issues arise: for example, 
		   	\eqref{eq:condequiv} and \eqref{eq:rcpdprop} need to be revisited.
		   	

	\end{remark}
	
\section{A second proof of Theorem \ref{theo:main}}
\label{sec:proof2}

In this section we provide a second proof
of~Theorem \ref{theo:main}. Compared to the above first
proof, the technical details are lighter.
However we emphasize that it uses that weak uniqueness
for Brownian SDEs is equivalent to uniqueness of  solutions
to the corresponding martingale problems. Therefore
it may be more difficult to extend
this second methodology than the first one
to stochastic control problems where
the noise involves a Poisson
random measure.

	Our  second proof of Theorem \ref{theo:main} is based
on the notion of controlled martingale problems on the
enlarged canonical space
$\Omb := \Om^2$.
	Denote by $\Fb = (\Fcb_t)_{t \ge 0}$ the canonical filtration, and by $(B,X)$ the canonical process.

	Fix $(t,\nu) \in \R^+ \x\Uc$. Let $\bar{b}^{t,\nu}$ and
$\bar{\sigma}^{t,\nu}$ be defined as in Section~\ref{subsec:ctrl_SDE}.
	For all functions $\varphi$ in $\mathcal{C}^2_c(\R^{2d})$,
	let the process $(\overline M^{t,\nu,\varphi}_{\theta}, \theta \ge t)$ be defined on the enlarged space $\Omb$ by
	\begin{equation} \label{eq:Mb}
		\overline M_\theta^{t,\nu,\varphi}(\omb)
		~~:=~~
		\varphi(\omb_{\theta})
		~-~
		\int_t^\theta \Lc^{t,\nu}_s
		\varphi(\omb) ds,~~\theta \ge t,
	\end{equation}
	where $\Lc^{t,\nu}$ is the differential operator
	\begin{equation*}
		\Lc^{t,\nu}_s \varphi(\omb)
		~~:=~~
		\bar{b}^{t,\nu}(s,\omb)\cdot D \varphi(\omb_s)
		~+~
		\dfrac{1}{2} ~\bar{a}^{t,\nu} (s,\omb) : D^2 \varphi(\omb_s);
	\end{equation*}
	here, we set $\bar{a}^{t,\nu}(s,\omb) := \bar{\sigma}^{t,\nu}(s,\omb){\bar{\sigma}^{t,\nu}(s,\omb)}^\ast$
	and the operator ``$:$'' is defined by $p:q := \text{Tr} (pq^*)$   for all $p$ and $q$ in $S_{2d,d}$.
	It is clear that the process $\overline M^{t,\nu,\varphi}$ is $\Fb$--progressively measurable.

	Let $(t,\xb,\nu) \in \R^+ \x \Om \x \Uc$. Denote by $\Pb^{t,\xb,\nu}$ the probability measure on $\Omb$ induced by $(B, X^{t,\xb,\nu})$ under the Wiener measure $\P$.
	For all $\varphi$ in $\mathcal{C}_c^2(\R^{2d})$, the process
	$\overline M^{t,\nu,\varphi}$ is a martingale under $\Pb^{t,\xb,\nu}$ and
	\begin{equation*}
		\Pb^{t,\xb,\nu} \big( X_s = \xb(s),~ 0 \le s \le t \big)
		~=~
		1.
	\end{equation*}
	Let $\tau$ be a $\F^{\P}$--stopping time taking values in $[t,+\infty)$ and let $\eta$ be the $\F$--stopping time
    defined in Lemma~\ref{lemma:tau_eta}.
	Set $\bar \nu (\wb) := \nu (\w)$ and
	$\etab (\wb) := \eta(\w)$ for all $\wb=(\w, \w') \in \Omb$.
	It is clear that $\etab$ is a $\Fb$--stopping time.
	Then there is a family of r.c.p. of $\Pb^{t,\xb,\nu}$ given $\Fcb_{\etab}$
	denoted by $(\Pb^{t,\xb,\nu}_{\wb})_{\wb \in \Omb}$, and a $\Pb^{t,\xb,\nu}$--null set $\overline N \subset \Omb$ such that
	\begin{equation}\label{eq:rcpd2}
		\Pb^{t,\xb,\nu}_{\wb} \big( \bar \eta = \eta(\w), B_s = \w(s),  X_s = \w'(s), 0 \le s \le \eta(\w) \big)
		~=~
		1
	\end{equation}
	for all $\wb = (\w, \w') \in \Omb \setminus \overline N$.
	In particular, one has
	\begin{equation*}
		\Pb^{t,\xb,\nu}_{\wb} \big( \bar \nu_s = \nu^{\eta(\w),\w}_s,\ \forall\, s\geq0 \big)
		~=~
		1
	\end{equation*}
	for all $\wb = (\w, \w') \in \Omb \setminus \overline N$. Moreover, Lemma 6.1.3 in~\cite{stroock-varadhan}
	combined with a standard localization argument
implies that for $\Pb^{t,\xb,\nu}$--a.a. $\wb\in\Omb$, for all $\varphi \in C_c^2(\R^{2d})$, the process
	\begin{equation*}
		\varphi(\omb_{\theta})
		~-~
		\int_{\eta(\w)}^\theta \Lc^{t,\nu}_s
		\varphi(\omb) ds,~~\theta \ge \eta(\w),
	\end{equation*}
	is a martingale under $\Pb^{t,\xb,\nu}_{\wb}$. It follows by \eqref{eq:rcpd2} that  for $\Pb^{t,\xb,\nu}$--a.a. $\wb\in\Omb$, for all $\varphi \in C_c^2(\R^{2d})$, $\overline M^{\eta(\w),\nu^{\eta(\w),\w},\varphi}$ is a martingale under $\Pb^{t,\xb,\nu}_{\wb}$.
	
	As weak uniqueness is equivalent to uniqueness of solutions to martingale problem\footnote{
	Here the SDE to consider is : for all $\theta\in\R^+$,
	\begin{equation*}
		 Z_{\theta}
		 ~=~
		 \wb(\eta(\w)\wedge\theta)
		 +
		 \int_{\eta(\w)}^{\eta(\w)\vee\theta} \bar{b}^{\eta(\w),\nu^{\eta(\w),\w}}(s,Z)ds
		 +
		 \int_{\eta(\w)}^{\eta(\w)\vee\theta} \bar{\sigma}^{\eta(\w),\nu^{\eta(\w),\w}}(s,Z)dB_s.
	 \end{equation*}
	}, for $\Pb^{t,\xb,\nu}$--a.a. $\wb\in\Omb$,
	$\Pb^{t,\xb,\nu}_{\wb}$ coincides with the probability measure on $\Omb$ induced by
	$\big( \w \ox_{\eta(\w)} B, X^{\eta(\w), \w', \nu^{\eta(\w),\w}} \big)$
	under the Wiener measure $\P$.
	Therefore, for all bounded random variable $Y$ in $\Fc_{\eta}$ we have
	\begin{eqnarray*}
		\E^{\P} \left[ \Phi \left(X^{t,\xb,\nu}\right) Y\right]
		 &=&
		\int_{\Omb} {\left(\Phi(\w')Y(\w)\right)} \,\Pb^{t,\xb,\nu}(d\wb) \\
		&=&
		\int_{\Omb} {\left(\E^{\Pb^{t,\xb,\nu}_{\wb}}[\Phi(X)] ~Y(\w)\right)} \,\Pb^{t,\xb,\nu}(d\wb)\\
		&=&
		\int_{\Omb} {\left(\E^{\P}\left[\Phi\left(X^{\eta(\w),\w',\nu^{\eta(\w),\w}}\right)\right] Y(\w)\right)} \,\Pb^{t,\xb,\nu}(d\wb )\\
		&=&
		\int_{\Om} { \left( J \left(\eta(\om),X^{t,\xb,\nu}(\om),\nu^{\eta(\om),\om}\right) Y(\om) \right) } \,\P(d\om).
	\end{eqnarray*}
	Since $\eta=\tau~\P-\text{a.s.}$, we have completed the proof.
	\qed

\end{document}